\documentclass[a4paper,11pt]{amsart}

\usepackage[utf8]{inputenc}

\usepackage{amssymb}

\makeatletter
\@namedef{subjclassname@2010}{%
  \textup{2010} Mathematics Subject Classification}
\makeatother

\usepackage[T1]{fontenc}

\allowdisplaybreaks[2]

\newtheorem{theorem}{Theorem}[section]
\newtheorem{proposition}[theorem]{Proposition}
\newtheorem{lemma}[theorem]{Lemma}
\newtheorem{corollary}[theorem]{Corollary}
\theoremstyle{definition}
\newtheorem{remark}[theorem]{Remark}

\newtheorem{example}[theorem]{Example}
\newtheorem{question}[theorem]{Question}
\numberwithin{equation}{section}

\begin{document}

\title[Zero Lie product determined Banach algebras, II]{Zero Lie product determined Banach algebras, II}

\author{J. Alaminos}
\address{Departamento de An\' alisis
Matem\' atico\\ Fa\-cul\-tad de Ciencias\\ Universidad de Granada\\
18071 Granada, Spain}
\email{alaminos@ugr.es}
\author{M. Bre\v sar}
\address{Faculty of Mathematics and Physics\\
University of Ljubljana\\
Jadranska 19, 1000 Ljubljana,
 and \\
Faculty of Natural Sciences and Mathematics\\
University of Maribor\\
Koro\v ska 160, 2000 Maribor, Slovenia} \email{matej.bresar@fmf.uni-lj.si}
\author{J. Extremera}
\address{Departamento de An\' alisis
Matem\' atico\\ Fa\-cul\-tad de Ciencias\\ Universidad de Granada\\
18071 Granada, Spain}
\email{jlizana@ugr.es}
\author{A.\,R. Villena}
\address{Departamento de An\' alisis
Matem\' atico\\ Fa\-cul\-tad de Ciencias\\ Universidad de Granada\\
18071 Granada, Spain}
\email{avillena@ugr.es}

\date{}

\begin{abstract}
A Banach algebra $A$ is said to be zero Lie product determined if
every continuous bilinear functional
$\varphi \colon A\times A\to \mathbb{C}$ satisfying
$\varphi(a,b)=0$ whenever $ab=ba$
 is of the form $\varphi(a,b)=\omega(ab-ba)$ for some $\omega\in A^*$.
We prove that $A$
has this property  provided that  any of the following three conditions holds:
(i) $A$ is a weakly amenable Banach algebra with property $\mathbb{B}$ and having a bounded approximate identity, (ii) every continuous cyclic Jordan derivation from $A$ into $A^*$ is an inner derivation, (iii) $A$ is the algebra of all $n\times n$ matrices, where $n\ge 2$, over a cyclically amenable Banach algebra with a bounded approximate identity.
\end{abstract}


\thanks{
The authors were supported by MINECO grant MTM2015--65020--P.
The first, the third  and the fourth named authors were supported
by  Junta de Andaluc\'{\i}a grant FQM--185.
The second  named author was supported by ARRS grant P1--0288.}

\maketitle

\section{Introduction}

Let $A$ be a Banach algebra. For $a$, $b\in A$, we will write
$[a,b]=ab-ba$ and $a\circ b = ab+ba$.
Let $\varphi \colon A\times A\to\mathbb{C}$ be a continuous bilinear functional satisfying
\begin{equation}\label{B}
a,b\in A, \ [a,b]=0 \ \Longrightarrow \ \varphi(a,b)=0.
\end{equation}
This is certainly fulfilled if $\varphi$ is of the form
\begin{equation}\label{Bl}
\varphi(a,b)=\omega([a,b])  \quad (a,b\in A)
\end{equation}
for some $\omega$ in $A^*$, the dual of $A$. We  say that $A$ is  a
\emph{zero Lie product determined Banach algebra} if, for every continuous bilinear functional
$\varphi \colon A\times A\to\mathbb{C}$ satisfying~\eqref{B}, there exists $\omega\in A^*$ such that~\eqref{Bl} holds.
This notion, introduced in our recent paper~\cite{ABEVi},
can be viewed as the  Lie  analogue of
the notion of a Banach algebra with property $\mathbb B$, introduced in~\cite{ABEV0} and subsequently studied in several papers
(see~\cite{ABESV} and references therein).

The main result of~\cite{ABEVi} states that the group algebra $L^1(G)$, where
$G$ is an \emph{amenable} locally compact group,  is zero Lie product determined. In
Section~\ref{s2} we take a different approach and  in particular show that this is actually
true for \emph{every} locally compact group $G$. The
main  result (Theorem~\ref{t1}) shows much more, namely, that
every  weakly amenable Banach algebra with property  $\mathbb{B}$ and
 a bounded approximate identity satisfies a certain generalized version of the zero Lie product determined property. Yet another approach is taken in
Section~\ref{s3}, connecting
the problem of showing that a Banach algebra is zero Lie product determined  with the problem of showing that a Jordan derivation is an (inner) derivation.
It is interesting to note that each of the two approaches yields the fact that $C^*$-algebras are zero Lie product determined. This was also observed in~\cite{ABEVi}, but derived as a simple corollary to a result by
Goldstein~\cite{G}. Finally, in Section~\ref{s4}, we use the  result of
Section~\ref{s3} to show that
if $A$ is a  cyclically amenable Banach algebra with a bounded approximate identity, then
$M_n(A)$,
the Banach algebra of $n\times n$ matrices (with $n\ge 2$) over  $A$, is zero Lie product determined.

\section{Weak amenability and property $\mathbb{B}$}\label{s2}
Recall that a Banach algebra $A$ is said to be \emph{weakly amenable}
if every continuous derivation from $A$ into $A^*$ is inner.
For a thorough treatment of this property and an account of many interesting examples of
weakly amenable Banach algebras we refer the reader to~\cite{D}.
We should remark that each $C^*$-algebra and the group algebra $L^1(G)$ of each locally compact group
$G$ are weakly amenable~\cite[Theorems 5.6.48 and 5.6.77]{D}.
If the Banach algebra $A$ is commutative, then $A^*$ is a commutative Banach $A$-bimodule and therefore
a derivation from $A$ into $A^*$ is inner if and only if it is zero.
It is worth pointing out that a basic obstruction to the weak amenability
is the existence of non-zero, continuous point derivations~\cite[Theorem~2.8.63(ii)]{D}. Recall that a linear functional $D$ on $A$
is a \emph{point derivation at} a given multiplicative linear functional $f$ if
\begin{equation}\label{pd}
D(ab)=D(a)f(b)+f(a)D(b) \quad  (a,b\in A).
\end{equation}

We say that a Banach algebra $A$ has \emph{property $\mathbb{B}$} if for every continuous bilinear functional
$\varphi \colon A\times A\to \mathbb{C}$, the condition
\begin{equation}\label{B1}
a,b\in A, \ ab=0 \  \Longrightarrow \  \varphi(a,b)=0
\end{equation}
implies the condition
\begin{equation}\label{B11}
\varphi(ab,c)=\varphi(a,bc)  \quad (a,b,c\in A).
\end{equation}
In~\cite{ABEV0} it was shown that many important examples of Banach algebras,
including $C^*$-algebras and group algebras $L^1(G)$ where $G$ is any locally compact group,
have property $\mathbb{B}$.

For a Banach space $X$, we denote by $\mathcal{B}(X)$ the Banach algebra of bounded
linear operators on $X$,  and  by $\mathcal{A}(X)$ the closed ideal of $\mathcal{B}(X)$
consisting of approximable operators.

\begin{proposition}\label{pr01}
Let $X$ be a Banach space. Then:
\begin{enumerate}
\item
The Banach algebra $\mathcal{A}(X)$ has property $\mathbb{B}$;
\item
The Banach algebra $\mathcal{B}(X^n)$ has property $\mathbb{B}$ for each $n\in\mathbb{N}$ with $n\ge 2$.
\end{enumerate}
\end{proposition}

\begin{proof}
From~\cite[Example 1.3.3 and Theorem 2.11]{ABEV0} we see that $\mathcal{A}(X)$  has property $\mathbb{B}$.
The Banach algebra $\mathcal{B}(X^n)$ is isomorphic to
the Banach algebra $M_n(\mathcal{B}(X))$ of all $n\times n$ matrices with entries in $\mathcal{B}(X)$.
By~\cite[Proposition~4.2]{ABESV},  $M_n(\mathcal{B}(X))$ is generated by idempotents
and~\cite[Example~1.3.3 and Theorem~2.11]{ABEV0} then show that it has property $\mathbb{B}$.
\end{proof}

\begin{remark}\label{r903}
Given a Banach space $X$, and $1\le p\le\infty$,
we write $\ell^p(X)$ for the Banach space of sequences $(x_n)$ in $X$
so that $(\Vert x_n\Vert)\in\ell^p$ with norm given by
$\Vert (x_n)\Vert=\Vert(\Vert x_n\Vert)\Vert_p$.
Since $\ell^p(X)$ is isomorphic to $\ell^p(X)\oplus \ell^p(X)$,
it follows that the Banach algebra $\mathcal{B}(\ell^p(X))$ has property $\mathbb{B}$.
In particular, the Banach algebra $\mathcal{B}(\ell^p)$ has property $\mathbb{B}$.
\end{remark}

\begin{proposition}\label{pr02}
Let $A$ be a Banach algebra.
\begin{enumerate}
\item
If $A$ is essential and has property $\mathbb{B}$,
then there are no non-zero, continuous point derivations on $A$.
\item
If $A$ is unital and contains a closed one-codimensional ideal $I$
such that $I^2$ is not dense in $I$,  then there is a non-zero, continuous point derivation on $A$.
\end{enumerate}
\end{proposition}

\begin{proof}
\begin{enumerate}
\item  Since $A$ is essential there are no non-zero, continuous point derivations on $A$ at $0$.
Let $D\in A^*$ be a continuous point derivation at a character $f$ of $A$.
We define a continuous bilinear functional $\varphi\colon A\times A\to\mathbb{C}$ by
$\varphi(a,b)=f(a)D(b)$ for all $a$, $b\in A$.
We claim that $\varphi$ satisfies~\eqref{B1}. Suppose that $a$, $b\in A$ are such that
$ab=0$.
Then $f(a)f(b)=0$.
If $f(a)=0$, then clearly $\varphi(a,b)=0$, as claimed.
If $f(a)\ne 0$, then $f(b)=0$ and~\eqref{pd} yields $\varphi(a,b)=D(ab)-D(a)f(b)=0$, as claimed.
Let $a\in A$ be such that $f(a)=1$. Since $A$ has property~$\mathbb{B}$, \eqref{B11}  holds and we have
\begin{align*}
f(b)D(a) & = f(ab)D(a)=\varphi(ab,a) = \varphi(a,ba) \\
& =f(a)D(ba)=D(ba)=D(b)+f(b)D(a),
\end{align*}
which gives $D(b)=0$ for each $b\in A$.

\item
Let $f$ be a character of $A$ such that $\ker(f)=I$,
and let $x\in I$ be such that $x\not\in \overline{I^2}$.
Then $M=\mathbb{C}\mathbf{1}+\overline{I^2}$ is a closed subspace of $A$ and $x\not\in M$.
Hence there exists $D\in A^*$ such that $D(M)=\{0\}$ and $D(x)=1$.
By~\cite[Proposition~2.7.11(i)]{D}, $D$ is a non-zero, continuous point derivation at $f$.
\end{enumerate}
\end{proof}

\begin{corollary}
There exists a Banach space $X$ such that $\mathcal{B}(X)$ does not have property  $\mathbb{B}$.
\end{corollary}

\begin{proof}
An example of a Banach space $X$ such that $\mathcal{B}(X)$
contains an ideal $I$ satisfying  the conditions of the second assertion of Proposition~\ref{pr02}
is given in~\cite{R}. Thus there exists a non-zero, continuous point derivation on $\mathcal{B}(X)$.
The first assertion of Proposition~\ref{pr02} then establishes the result.
\end{proof}

The following result is stated  in \cite{ABEVJordan} as Theorem 2.2, but only for the case where $A$ is a $C^*$-algebra. From the proof, however, it is clear that it holds
for every Banach algebra with property $\mathbb B$ that has a bounded approximate identity. See also~\cite[Lemma 4.1]{ABEVi} which covers the general case, only
the (easy) part concerning the functional $\tau$ is missing.

\begin{lemma}\cite{ABEVJordan}\label{l1}
Let $A$ be a Banach algebra with property $\mathbb{B}$ and having a bounded approximate identity,
  and
let $\varphi\colon A\times A\to \mathbb C$ be a continuous bilinear map satisfying the condition:
\begin{equation*}
a,b\in A, \ ab=ba=0 \ \Longrightarrow \ \varphi(a,b)=0.
\end{equation*}
Then there exists  $\sigma \in  A^*$  such that
\begin{equation}\label{eqfifi}
\varphi(ab,c)-\varphi(b,ca)+\varphi(bc,a)=\sigma(abc) \quad (a,b,c\in A).
\end{equation}
Moreover, there exists $\tau\in A^*$  such that
\[
\varphi(a,b) + \varphi(b,a)= \tau(a\circ b)  \quad (a,b\in A).
\]
\end{lemma}

In the next lemma we take the conclusion of Lemma~\ref{l1} as an assumption. But first, recall that given any Banach algebra $A$, the dual space $A^*$ of $A$ becomes a Banach $A$-bimodule by setting
\[
(a\cdot f)(c)= f(ca),\quad (f\cdot a)(c)= f(ac) \quad (a,c\in A,\ f\in A^*).
\]

\begin{lemma}\label{l2}
Let $A$ be any Banach algebra. If a skew-symmetric continuous bilinear map $\varphi\colon A\times A\to \mathbb C$ satisfies~\eqref{eqfifi} for some   $\sigma\in A^*$, then the map $\delta\colon A\to A^*$ defined by
\[
\delta(a)(c) =\varphi(a,c) + \frac{1}{2}\sigma(a\circ c)
\]
is a continuous derivation.
\end{lemma}

\begin{proof}
It is clear that $\delta$ is linear and continuous. We must show that it satisfies
\begin{equation}\label{der}
\delta(ab)=\delta(a)\cdot b + a\cdot\delta(b)\quad (a,b\in A).
\end{equation}
We begin by noticing that~\eqref{eqfifi} can be written as
\begin{equation} \label{e23}
\sigma(abc)=\varphi(ab,c)+\varphi(ca,b)+\varphi(bc,a) \quad (a,b,c\in A)\end{equation}
for $\varphi$ is skew-symmetric.
Note that this immediately yields
\begin{equation} \label{sig}
\sigma(abc)= \sigma(bca) \quad (a,b,c\in A).
\end{equation}
Now take arbitrary $a,b,c\in A$. We have
\[
\delta(ab)(c) = \varphi(ab,c) +\frac{1}{2}\sigma(ab\circ c).
\]
From~\eqref{e23} it follows that
\[
\delta(ab)(c) = - \varphi(ca,b) - \varphi(bc,a) + \sigma(abc)  +\frac{1}{2}\sigma(ab c + cab),
\]
and hence, in view of~\eqref{sig} and the skew-linearity of $\varphi$,
\begin{equation}\label{abc}
\delta(ab)(c) =  \varphi(b,ca) + \varphi(a,bc) + \sigma(bca) +\frac{1}{2}\sigma(ab c + cab).
\end{equation}
Now compute the right-hand side of~\eqref{der}:
\begin{align*}
\big(\delta(a)\cdot b + a\cdot\delta(b)\big)(c)&= \delta(a)(bc) +
\delta(b)(ca)\\
&=\varphi(a,bc) + \frac{1}{2}\sigma(a\circ bc) + \varphi(b,ca) +\frac{1}{2}\sigma(b\circ ca).
\end{align*}
Comparing this with~\eqref{abc}, we see that~\eqref{der} indeed holds.
\end{proof}

As an easy consequence of Lemmas~\ref{l1} and~\ref{l2} we can now derive the following theorem.

\begin{theorem}\label{t1}
Let $A$ be a  weakly amenable Banach algebra with property $\mathbb{B}$ and having a bounded approximate identity.
If a continuous bilinear functional $\varphi \colon A\times A\to\mathbb{C}$ satisfies
\begin{equation*}
a,b\in A, \ ab=ba=0 \ \Longrightarrow \ \varphi(a,b)=0,
\end{equation*}
then  there exist $\tau_1,\tau_2\in A^*$ such that
\[
\varphi(a,b)= \tau_1(ab) + \tau_2(ba)\quad (a,b\in A).
\]
\end{theorem}

\begin{proof}
Define $\varphi_1,\varphi_2\colon A\times A\to\mathbb C$ by
\[
\varphi_1(a,b) = \frac{1}{2}\big(\varphi(a,b) + \varphi(b,a)\big),\quad \varphi_2(a,b) = \frac{1}{2}\big(\varphi(a,b) - \varphi(b,a)\big).
\]
Note that $\varphi = \varphi_1 + \varphi_2$, and each $\varphi_i$ satisfies the condition of the theorem, i.e., $ab=ba=0$ implies $\varphi_i(a,b)=0$. Since $\varphi_1$ is symmetric it follows from Lemma~\ref{l1} that
\begin{equation}\label{mm}
2\varphi_1(a,b)=\tau(a\circ b)\quad (a,b\in A)
\end{equation}
for some $\tau\in A^*$.  It remains to consider $\varphi_2$ which is  skew-symmetric.  By Lemma~\ref{l2} there exist $\sigma\in A^*$ and  a continuous derivation $\delta\colon A\to A ^*$ such that
\[
\delta(a)(c)= \varphi_2(a,c) + \frac{1}{2}\sigma(a\circ c) \quad (a,c\in A).
\]
Since $A$ is weakly amenable, $\delta$ is inner. That is, $\delta(a)=a\cdot \omega - \omega\cdot a$ for some $\omega\in A^*$, and hence
\[
\omega(ca-ac)-  \varphi_2(a,c) = \frac{1}{2}\sigma(a\circ c)\quad (a,c\in A).
\]
Viewing this expression as a bilinear functional on $A\times A$,  we see that the left-hand side is skew-symmetric and the right-hand side is symmetric. Therefore, both sides are zero.
Thus \begin{equation}\label{mm2}\varphi_2(a,c)= \omega(ca-ac)\quad (a,c\in A).\end{equation}
From~\eqref{mm} and~\eqref{mm2} we readily get the desired conclusion (with $\tau_1 = \frac{\tau}{2} - \omega$ and $\tau_2 = \frac{\tau}{2} + \omega$).
 \end{proof}

\begin{corollary}\label{wa}
Let $A$ be  a weakly amenable Banach algebra $A$ with property $\mathbb{B}$ and having  a bounded approximate identity.
Then $A$ is  a zero Lie product determined Banach algebra.
\end{corollary}

\begin{proof}
If $\varphi$ satisfies~\eqref{B}, then it is skew-symmetric
and satisfies the assumption of Theorem~\ref{t1}. This  clearly
implies that $\varphi$ is of the desired form.
\end{proof}

Every $C^*$-algebra satisfies all the assumptions of Corollary~\ref{wa}.
We can thus state  the following corollary which was derived
in~\cite{ABEVi} from Goldstein's theorem~\cite{G}.

\begin{corollary}\label{wac} Every $C^*$-algebra is a  zero Lie
product determined Banach algebra.
\end{corollary}

Another type of a Banach algebra satisfying all the assumptions
of Corollary~\ref{wa} is $L ^1(G)$ where $G$ is any locally compact
group. Hence we have the following corollary. It is a substantial
generalization of~\cite[Theorem~4.5]{ABEVi} which covers only
amenable groups.

\begin{corollary}\label{wag} Let $G$ be a locally compact group. Then the group
algebra $L ^1(G)$ is a zero Lie product determined Banach algebra.
\end{corollary}

\begin{corollary}
Let $X$ be a Banach space, and $1\le p\le\infty$.
Then $\mathcal{B}(\ell^p(X))$ is a zero Lie product determined Banach algebra.
In particular, $\mathcal{B}(\ell^p)$ is a zero Lie product determined Banach algebra.
\end{corollary}

\begin{proof}
Proposition~\ref{pr01} and Remark~\ref{r903}
show that $\mathcal{B}(\ell^p(X))$ has property $\mathbb{B}$.
Further,~\cite[Proposition 3.2]{Bla2} shows that $\mathcal{B}(\ell^p(X))$ is weakly amenable.
Corollary~\ref{wa} establishes our assertion.
\end{proof}

\section{Cyclic amenability and Jordan derivations}\label{s3}
Let $A$ be a Banach algebra.
A derivation $D\colon A\to A^*$ is said to be \emph{cyclic} if
\[
D(a)(b)=-D(b)(a) \quad (a,b\in A).
\]
Clearly inner derivations are cyclic. The algebra $A$ is said to be \emph{cyclically amenable}
if every continuous cyclic derivation $D\colon A\to A^*$ is inner.
For an account of this notion we refer the reader to~\cite{Gr1}.
Of course, weak amenability implies cyclic amenability.

Let $A$ be a unital Banach algebra that is unitally polynomially generated by a single element $x$,
and let $D\colon A\to A^*$ be a continuous cyclic derivation.
We have
\[
D(\mathbf{1})=D(\mathbf{1}^2)=2\bigl(\mathbf{1}\cdot D(\mathbf{1})\bigr)=2D(\mathbf{1}),
\]
which implies that $D(\mathbf{1})=0$.
Further, $D(x)(\mathbf{1})=-D(\mathbf{1})(x)=0$ and, for $n\in\mathbb{N}$,
\[
D(x)(x^n)=-D(x^n)(x)=-\bigl(nx^{n-1}\cdot D(x)\bigr)(x)=-nD(x)(x^n),
\]
which yields $D(x)(x^n)=0$. This clearly implies that
$D(x)(p(x))=0$ for each polynomial $p$ in one indeterminate,
and hence $D(x)(a)=0$ for each $a\in A$.
On the other hand, for each polynomial $p$ and each $a\in A$,
we have
\[
D\bigl(p(x)\bigr)(a)=
\bigl(p'(x)\cdot D(x)\bigr)(a)=
D(x)\bigl(p'(x)a\bigr)=0,
\]
and hence $D(b)(a)=0$ for each $b\in A$. Consequently, every cyclic
derivation $D\colon A\to A^*$ is zero and therefore $A$ is cyclically
amenable. Some relevant examples of unital Banach algebras that are
unitally polynomially generated by a single element are the disc algebra
$A(\mathbb{D})$, $C^n([a,b])$, and the algebra $AC([a,b])$ of absolutely
continuous functions on the interval $[a,b]$ (see~\cite{D}).
It should be pointed out that none of these Banach algebras is weakly amenable.
The Banach algebras $A(\mathbb{D})$ and $C^n([a,b])$ obviously have
non-zero, continuous point derivations and so they are not weakly amenable
(see~\cite[Theorem~2.8.63]{D}). There are no non-zero
continuous point derivations on $AC([a,b])$ but nevertheless this
Banach algebra is not weakly amenable~\cite[Theorem~5.6.8]{D}.
The semigroup algebra $\ell^1(\mathbb{S}_X)$ of the free semigroup
$\mathbb{S}_X$ generated by any set
$X$ is cyclically amenable~\cite{Gr1}. For the definition of the semigroup algebra we refer the reader to~\cite[Examples 2.1.13(iv),(v)]{D}.
However, the algebra $\ell^1(\mathbb{S}_X)$ is not essential.
For, if $x\in X$, then there are no elements $s,t\in\mathbb{S}_X$ with $st=x$,
which implies that $(a\star b)(x)=0$ for all $a,b\in\ell^1(\mathbb{S}_X)$.
From~\cite[Theorem 2.8.63(i)]{D} it follows that
the algebra $\ell^1(\mathbb{S}_X)$ is never weakly amenable, regardless of $X$.

Let $A$ be an algebra and $M$ an $A$-bimodule. For $a\in A$ and $m\in M$ we write $a\circ m = m\circ a = a\cdot m + m\cdot a$. Recall that a
 linear map $\delta\colon A\to M$  is called a Jordan derivation if it satisfies
\begin{equation}\label{j1}
\delta(a\circ b) = \delta(a)\circ b  + a \circ \delta(b)\quad (a,b\in A).
\end{equation}
By the usual polarisation procedure,~\eqref{j1} is equivalent to
\begin{equation}\label{j2}
\delta(a^2)=\delta(a)\cdot a+a\cdot\delta(a)\quad (a\in A).
\end{equation}
Taking into account that, for all $a$, $b\in A$ and $m\in M$,
\[
2aba=(a\circ b)\circ a-a^2\circ b
\]
and
\[
2a\cdot m\cdot a=(a\circ m)\circ a-a^2\circ m,
\]
it follows that every  Jordan derivation $\delta$ satifies
\begin{equation}\label{j3}
\delta(aba)=\delta(a)\cdot(ba)+a\cdot\delta(b)\cdot a+(ab)\cdot\delta(a).
\end{equation}

The question whether every Jordan derivation is actually a derivation has been studied by many authors in various contexts, starting in the 1950's.
Roughly speaking, more often than not the answer is positive (see, e.g.,~\cite{Br}). The following result is therefore of some interest.

\begin{theorem}\label{tja}
Let $A$ be a Banach algebra. If every continuous cyclic Jordan derivation from $A$
into $A^*$ is an inner derivation, then $A$ is a zero Lie product determined Banach algebra.
\end{theorem}

\begin{proof}
Let $\varphi\colon A\times A\to \mathbb C$ be a  continuous bilinear functional satisfying~\eqref{B}. In particular, $\varphi$ satisfies
\[
\varphi(a,a)=0\quad\mbox{and}\quad \varphi(a^2,a)=0\quad (a\in A).
\]
Linearizing these identities we see that $\varphi$ is skew-symmetric and satisfies
\begin{equation}\label{Jo}
\varphi(a\circ b,c) + \varphi(c\circ a,b) + \varphi(b\circ c,a) =0\quad (a,b,c\in A).
\end{equation}
Define $\delta\colon A\to A^*$ by
\[
\delta(a)(c)= \varphi(a,c)
\]
for all $a$, $c\in A$. Applying~\eqref{Jo}, for all $a$, $b$, $c\in A$ we obtain
\begin{align*}
\delta(a\circ b)(c)=\varphi(a\circ b,c)=  \varphi(a,b\circ c) +\varphi(b,c\circ a).
\end{align*}
From
\begin{align*}
\big(\delta(a)\circ b  + a \circ \delta(b)\big)(c) & = \big(\delta(a)\cdot b + b\cdot \delta(a) + a\cdot \delta(b) + \delta(b)\cdot a\big)(c)\\
& = \delta(a) (bc) + \delta(a)(cb) + \delta(b)(ca) + \delta(b)(ac)\\
& = \varphi(a,b\circ c) +\varphi(b,c\circ a)
\end{align*}
we thus see that $\delta$ is a Jordan derivation.
Further, it is clear that $\delta$ is cyclic. By our hypothesis, it follows that $\delta$ is an inner derivation.
Hence there exists $\omega\in A^*$ such that $\varphi(a,c)=\omega([c,a])$.
\end{proof}

\begin{corollary}\label{rtja}
Let $A$ be a cyclically amenable Banach algebra. If every
continuous Jordan derivation from $A$ into $A^*$ is a derivation, then  $A$ is a zero Lie product determined Banach algebra.
\end{corollary}

\begin{remark}
Johnson~\cite{J} proved that every continuous Jordan derivation from a $C^*$-algebra $A$ into any Banach $A$-bimodule $M$ is a derivation
(see~\cite{ABV} for an alternative proof). Therefore Corollary~\ref{wac} follows also from Corollary~\ref{rtja}.
In~\cite{J} it is also shown that every continuous Jordan derivation from the group algebra $L^1(G)$
into any Banach $L^1(G)$-bimodule $M$ is a derivation in the case where the group $G$ is amenable.
This clearly implies that Corollary~\ref{wag} follows from  Corollary~\ref{rtja} in the case where the group is amenable.
Nevertheless, to the best of our knowledge, it is not known whether every continuous Jordan derivation from
$L^{1}(G)$ into $L^{1}(G)^{*}$ is actually a derivation for each locally compact group $G$.
\end{remark}

\begin{proposition}\label{p2044}
Let $\mathbb{S}_2$ be the free semigroup with two generators.
Then there exists a continuous Jordan derivation
from $\ell^1(\mathbb{S}_2)$ into $\ell^1(\mathbb{S}_{2})^*$ which is not a derivation.
\end{proposition}

\begin{proof}
Let $x_1$ and $x_2$ be the generators of $\mathbb{S}_2$, and
let $u_1,u_2\in M_4(\mathbb{C})$ be the matrices given by
\[
u_1=
\begin{pmatrix}
0&1&0&0\\
0&0&0&0\\
0&0&0&-1\\
0&0&0&0\\
\end{pmatrix}, \
u_2=
\begin{pmatrix}
0&0&1&0\\
0&0&0&1\\
0&0&0&0\\
0&0&0&0\\
\end{pmatrix}.
\]
For each word $w\in \mathbb{S}_2$
let $\delta_w\in \ell^1(\mathbb{S}_2)$ stands for the characteristic function of the set $\{w\}$.
Further, each element $w\in\mathbb{S}_2$ has the form
$w=x_{j_1}^{k_1}\cdots x_{j_n}^{k_n}$ where $n, k_1,\ldots,k_n\in\mathbb{N}$ and $j_1,\ldots,j_n\in \{1,2\}$,
and, since $\Vert u_1\Vert=\Vert u_2\Vert=1$, there exists a unique continuous algebra homomorphism
$\Phi\colon\ell^{1}(\mathbb{S}_2)\to M_4(\mathbb{C})$ such that
\[
\Phi(\delta_{x_{j_1}^{k_1}\cdots x_{j_n}^{k_n}})=u_{j_1}^{k_1}\cdots u_{j_n}^{k_n}
\quad (n, k_1,\ldots,k_n\in\mathbb{N}, \, j_1,\ldots,j_n\in \{1,2\})
\]
Observe that $\Phi(\ell^1(\mathbb{S}_2))$ is the subalgebra of $M_4(\mathbb{C})$ generated by $u_1$ and $u_2$
and that
\[
u_1^2=u_2^2=u_1u_2+u_2u_1=0.
\]
This implies that $\Phi(a\circ b)=0$ for all $a,b\in\ell^1(\mathbb{S}_2)$.
We also define $\tau\colon M_4(\mathbb{C})\to\mathbb{C}$ by $\tau(u)=u_{12}+u_{14}$ $(u\in M_4(\mathbb{C}))$,
and finally
$D\colon\ell^{1} (\mathbb{S}_2)\to\ell^{1} (\mathbb{S}_{2})^*$ by
\[
D(a)(b)=\tau(\Phi(a))\tau(\Phi(b)) \quad (a,b\in\ell^1(\mathbb{S}_2)).
\]
We claim that $D$ is a Jordan derivation. If $a$, $b\in\ell^1(\mathbb{S}_2)$,
then $\Phi(a^2)=0$, which yields $D(a^2)(b)=0$, and
\[
\bigl(D(a)\cdot a+a\cdot D(a)\bigr)(b)=D(a)(a\circ b)=\tau(\Phi(a))\tau(\Phi(a\circ b))=0.
\]
We now show that $D$ is not a derivation. Indeed,
\[
D(\delta_{x_1}\star\delta_{x_2})(\delta_{x_1})=D(\delta_{x_1x_2})(\delta_{x_1})=1,
\]
while
\[
\bigl(D(\delta_{x_1})\cdot\delta_{x_2}+\delta_{x_1}\cdot D(\delta_{x_2})\bigr)(\delta_{x_1})=
D(\delta_{x_1})(\delta_{x_2x_1})+D(\delta_{x_2})(\delta_{x_1^2})=-1. \qedhere
\]
\end{proof}

\begin{question}
Is the semigroup algebra $\ell^1(\mathbb{S}_2)$ a zero Lie product
determined Banach algebra? Since $\ell^1(\mathbb{S}_2)$ is cyclically
amenable it suffices to find out whether every continuous
cyclic Jordan derivation from $\ell^1(\mathbb{S}_2)$ into
$\ell^1(\mathbb{S}_{2})^*$ is a derivation. Proposition~\ref{p2044} indicates the delicacy of this problem.
\end{question}

\begin{remark}
Let $X$ be a Banach space.
On account of~\cite[Theorem 6.4]{J}, every continuous Jordan derivation from $\mathcal{A}(X)$ into
any Banach $\mathcal{A}(X)$-bimodule is a derivation.
On the other hand,
we refer the reader to~\cite{Bla, DGG, Gr2, Gr3} for a deep discussion of the weak amenability of $\mathcal{A}(X)$.
It is worth pointing out that $\mathcal{A}(\ell^p(X))$ is weakly amenable for each Banach space $X$ having the
bounded approximation property and $1\le p<\infty$ (see~\cite[Corollary 4.3]{Gr3}).
\end{remark}

The preceding remark together with Theorem~\ref{tja} yield the following.

\begin{corollary}
Let $X$ be a Banach space having the bounded approximation property, and $1\le p<\infty$.
Then  $\mathcal{A}(\ell^p(X))$ is a zero Lie product determined Banach algebra.
In particular, $\mathcal{A}(\ell^p)$ is a zero Lie product determined Banach algebra.
\end{corollary}

\section{Matrix algebras}\label{s4}

Let $A$ be a Banach algebra, and let $n\in\mathbb{N}$.
By $M_n(A)$ we denote the Banach algebra of $n\times n$ matrices with entries in $A$.
Let $A^\sharp$ denote the algebra formed by adjoining an identity to $A$
as defined in~\cite[Definition~1.3.3]{D}.
For $i$, $j\in \{1,\ldots,n\}$ and $a\in A^\sharp$,
let  $aE_{ij}\in M_n(A^\sharp)$  be the matrix with $a$ in the  $(i,j)\text{th}$ position and $0$ elsewhere.
We abbreviate $\mathbf{1}E_{ij}$ to $E_{ij}$. Note that for every $\mathbf{a}=(a_{ij})\in M_n(A^\sharp)$
we have
\begin{equation}\label{mu1}
E_{ij}\mathbf{a}E_{kl}=a_{jk}E_{il} \quad  (1\le i,j,k,l\le n).
\end{equation}
The algebra $M_n(A)$ is a closed two-sided ideal of $M_n(A^\sharp)$ and therefore $M_n(A)$ is a Banach
$M_n(\mathbb{C}\mathbf{1})$-bimodule.
From now on, we identify $M_n(\mathbb{C}\mathbf{1})$ with $M_n(\mathbb{C})$ in the natural way.

\begin{lemma}\label{lj1}
Let $A$ be a Banach algebra with a bounded approximate identity, and let $D\colon A\to A^*$ be a continuous Jordan derivation. Suppose that $B$ is a Banach algebra which contains $A$ as a closed
two-sided ideal.
Then there exists a continuous Jordan derivation $\Delta\colon B\to A^*$ which extends $D$.
\end{lemma}

\begin{proof}
Let $(u_{\lambda})_{\lambda\in\Lambda}$ be an approximate identity for $A$ of bound $C$,
and let $\mathcal{U}$ be an ultrafilter on $\Lambda$ refining the order filter.
Let $a\in A$ and $x\in B$. For every $\lambda\in\Lambda$, we have
\[
\vert D(u_\lambda x u_\lambda)(a)\vert\le
\Vert D\Vert\Vert u_\lambda xu_\lambda\Vert\Vert a\Vert\le
C^2\Vert D\Vert \Vert x\Vert\Vert a\Vert.
\]
Therefore the net of complex numbers $\bigl(D(u_\lambda x u_\lambda)(a)\bigr)_{\lambda\in\Lambda}$
is bounded by $C^2\Vert D\Vert \Vert x\Vert\Vert a\Vert$ and so it has a unique limit along the ultrafilter $\mathcal{U}$.
Hence we can define $\Delta\colon B\to A^*$ by
\[
\Delta(x)(a)=\lim_\mathcal{U}D(u_\lambda x u_\lambda)(a)
\]
for all $x\in B$ and $a\in A$.
Indeed, it is routine that $\Delta(x)\in A^*$ with $\Vert \Delta(x)\Vert\le C^2\Vert D\Vert \Vert x\Vert$ and
that $\Delta$ is a continuous linear map with $\Vert\Delta\Vert\le C^2\Vert D\Vert$.

We claim that  $\Delta$ extends $D$.
If $b\in A$, then $(u_\lambda b u_\lambda)_{\lambda\in\Lambda}\to b$ in norm and the continuity
of $D$ gives $\bigl(D(u_\lambda b u_\lambda)\bigr)_{\lambda\in\Lambda}\to D(b)$.
We thus get
\[
\Delta(b)(a)=\lim_\mathcal{U}D(u_\lambda b u_\lambda)(a)=D(b)(a)
\]
for each $a\in A$.

Our next goal is to show that $\Delta$ is a Jordan derivation.
By~\eqref{j3} we have
\begin{equation*}
\begin{split}
D(u_\lambda x^2 u_\lambda)(bcb)
& =
\bigl(b\cdot D(u_\lambda x^2 u_\lambda)\cdot b\bigr)(c)\\
& =
D\bigl(b(u_\lambda x^2 u_\lambda)b\bigr)(c)\\
&\quad {}-\Bigl(D(b)\cdot\bigl((u_\lambda x^2 u_\lambda)b\bigr)\Bigr)(c)
-\Bigl(\bigl(b(u_\lambda x^2 u_\lambda)\bigr)\cdot D(b)\Bigr)(c)\\
& = D(bu_\lambda x^2 u_\lambda b)(c) -D(b)(u_\lambda x^2 u_\lambda bc)
-D(b)(cbu_\lambda x^2 u_\lambda).\\
\end{split}
\end{equation*}
Since $\Vert bu_\lambda- b\Vert\to 0$ and $\Vert u_\lambda b-b\Vert\to 0$,
it follows that $\Vert bu_\lambda x^2 u_\lambda b-bx^2b\Vert\to 0$, and the continuity
of $D$ then gives
\[
\lim_\mathcal{U}D(bu_\lambda x^2 u_\lambda b)(c)= D(bx^2b)(c).
\]
On the other hand,
\begin{equation*}
\begin{split}
\Vert u_\lambda x^2u_\lambda bc-x^2bc\Vert
&\le
\Vert u_\lambda x^2u_\lambda bc-u_\lambda x^2bc\Vert+
\Vert u_\lambda x^2 bc-x^2bc\Vert \\
&\le
\Vert u_\lambda x^2\Vert\Vert u_\lambda (bc)-bc\Vert+
\Vert u_\lambda (x^2 bc)-x^2bc\Vert \\
&\le
C\Vert x^2\Vert\Vert u_\lambda (bc)-bc\Vert+
\Vert u_\lambda (x^2 bc)-x^2bc\Vert\to 0,\\
\end{split}
\end{equation*}
and the continuity of the functional $D(b)$ now gives
\[
\lim_\mathcal{U}D(b)(u_\lambda x^2 u_\lambda bc)=D(b)(x^2bc).
\]
By a similar argument we prove that
\[
\lim_{\mathcal{U}} \Vert cbu_\lambda x^2 u_\lambda-cbx^2\Vert = 0
\]
and hence that
\[
\lim_\mathcal{U}\bigl(D(b)(cbu_\lambda x^2 u_\lambda)\bigr)_{\lambda\in\Lambda}=D(b)(cbx^2).
\]
We thus get
\[
\Delta(x^2)(bcb)=D(bx^2b)(c)-D(b)(x^2bc)-D(b)(cbx^2).
\]
In much the same way we get
\begin{equation}\label{j4}
\begin{split}
D\bigl((u_\lambda xu_\lambda)^2\bigr)(bcb) & =
D(u_\lambda xu_\lambda^2 x u_\lambda)(bcb) \\
& =
\bigl(b\cdot D(u_\lambda x u_\lambda^2 xu_\lambda)\cdot b\bigr)(c)\\
& =
D\bigl(b(u_\lambda x u_\lambda^2 x u_\lambda)b\bigr)(c)-\Bigl(D(b)\cdot\bigl((u_\lambda x u_\lambda^2 x u_\lambda)b\bigr)\Bigr)(c)\\
& \quad {}-\Bigl(\bigl(b(u_\lambda x u_\lambda^2 x u_\lambda)\bigr)\cdot D(b)\Bigr)(c)\\
& =
D(bu_\lambda x u_\lambda^2x u_\lambda b)(c)-D(b)(u_\lambda x u_\lambda^2 x u_\lambda bc)\\
& \quad{}-D(b)(cbu_\lambda x u_\lambda^2 x u_\lambda)\\
& \to
D(bx^2b)(c)-D(b)(x^2bc)-D(b)(cbx^2)\\
& = \Delta(x^2)(bcb).
\end{split}
\end{equation}
On the other hand, from~\eqref{j2} we deduce that
\begin{equation*}
\begin{split}
D\bigl((u_\lambda xu_\lambda)^2\bigr)(bcb) & =
\bigl(D(u_\lambda xu_\lambda)\cdot(u_\lambda xu_\lambda)+
(u_\lambda xu_\lambda)\cdot D(u_\lambda xu_\lambda)\bigr)(bcb)\\
& = D(u_\lambda x u_\lambda)
\bigl(u_\lambda xu_\lambda bcb+bcbu_\lambda xu_\lambda\bigr).\\
\end{split}
\end{equation*}
Write
\[
v_\lambda=u_\lambda xu_\lambda bcb+bcbu_\lambda xu_\lambda
\]
and
\[
v=xbcb+bcbx.
\]
Then $(v_\lambda)_{\lambda\in\Lambda}\to v$ in norm and therefore
\begin{equation*}
\begin{split}
\vert D(u_\lambda x u_\lambda)(v_\lambda)-D(u_\lambda xu_\lambda)(v)\vert&\le
\Vert D(u_\lambda x u_\lambda)\Vert\Vert v_\lambda-v\Vert \\
&\le
C^2\Vert D\Vert\Vert x\Vert\Vert v_\lambda-v\Vert\to 0.
\end{split}
\end{equation*}
Since $\lim_\mathcal{U}D(u_\lambda xu_\lambda)(v)=\Delta(x)(v)$, it may be concluded that
\begin{equation*}
\begin{split}
\lim_\mathcal{U}D\bigl((u_\lambda xu_\lambda)^2\bigr)(bcb)
& =
\lim_\mathcal{U}D(u_\lambda x u_\lambda)(v_\lambda) =\Delta(x)(v)\\
& =\Delta(x)(xbcb+bcbx)\\
&=\bigl(\Delta(x)\cdot x+x\cdot\Delta(x)\bigr)(bcb).
\end{split}
\end{equation*}
From what has been proved and~\eqref{j4} we deduce that
\[
\bigl(\Delta(x^2)-\Delta(x)\cdot x-x\cdot\Delta(x)\bigr)(bcb)=0
\]
for all $b,c\in A$.
Let $a\in A$. Then~\cite[Th\'eor\`eme II.16]{AL} gives $b,c\in A$ such that $a=bcb$ and so
\[
\bigl(\Delta(x^2)-\Delta(x)\cdot x-x\cdot\Delta(x)\bigr)(a)=0.
\]
This entails that $\Delta(x^2)=\Delta(x)\cdot x+x\cdot\Delta(x)$.
\end{proof}

\begin{corollary}\label{cj2}
Let $A$ be a Banach algebra with a bounded approximate identity,
and let $D\colon M_n(A)\to M_n(A)^*$, where $n\ge 2$, be a continuous Jordan derivation.
Then $D$ is a derivation and there exists a continuous derivation $\Delta\colon M_n(A^\sharp)\to M_n(A)^*$ which extends $D$.
\end{corollary}

\begin{proof}
Let $(u_\lambda)_{\lambda\in\Lambda}$ be a bounded approximate identity for $A$ and,
for each $\lambda\in\Lambda$, take $\mathbf{u}_\lambda$ to be the matrix with $u_\lambda$ in
the diagonal and $0$ elsewhere.
Then $(\mathbf{u}_\lambda)_{\lambda\in\Lambda}$ is a bounded approximate identity for $M_n(A)$.
Further, $M_n(A)$ is a closed two-sided ideal of $M_n(A^\sharp)$. According to Lemma~\ref{lj1},
there exists a continuous Jordan derivation $\Delta\colon M_n(A^\sharp)\to M_n(A)^*$ which extends $D$.
By~\cite[Theorem 7.1]{J}, $\Delta$ is a derivation and this proves the result.
\end{proof}

\begin{lemma}\label{lj3}
Let $A$ be a cyclically amenable Banach algebra with a bounded approximate identity.
Then $M_n(A)$ is cyclically amenable.
\end{lemma}

\begin{proof}
Let $D\colon M_n(A)\to M_n(A)^*$ be a continuous cyclic derivation.

From Corollary~\ref{cj2} it follows that there exists a continuous derivation
$\Delta\colon M_n(A^\sharp)\to M_n(A)^*$ which extends $D$.
We now consider the restriction of $\Delta$ to the subalgebra $M_n(\mathbb{C})$ of $M_n(A^\sharp)$.
This gives a derivation from $M_n(\mathbb{C})$ to the Banach $M_n(\mathbb{C})$-bimodule $M_n(A)^*$.
Since every derivation from $M_n(\mathbb{C})$ to any Banach $M_n(\mathbb{C})$-bimodule is inner
(see~\cite[Example 1.9.23]{D}), we conclude that there exists a functional $F_1\in M_n(A)^*$ such that
\[
\Delta(\alpha)=\alpha\cdot F_1-F_1\cdot\alpha \quad  (\alpha\in M_n(\mathbb{C})),
\]
\emph{i.e.},
\[
\Delta(\alpha)(\mathbf{a})=F_1(\mathbf{a}\alpha-\alpha\mathbf{a}) \quad
(\alpha\in M_n(\mathbb{C}), \, \mathbf{a}\in M_n(A)).
\]

The map $D_0\colon M_n(A)\to M_n(A)^*$ defined by
\begin{equation}\label{pp2}
D_0(\mathbf{a})=D(\mathbf{a})-(\mathbf{a}\cdot F_1-F_1\cdot\mathbf{a}) \quad  (\mathbf{a}\in M_n(A))
\end{equation}
is a continuous cyclic derivation.
We claim that
\begin{equation}\label{pp1}
D_0(\alpha \mathbf{a}\beta)=\alpha D_0(\mathbf{a})\beta
\quad  (\alpha,\beta\in M_n(\mathbb{C}), \, \mathbf{a}\in M_n(A)).
\end{equation}
Indeed, if $\alpha,\beta\in M_n(\mathbb{C})$ and $\mathbf{a}\in M_n(A)$, then
\begin{equation*}
\begin{split}
D_0(\alpha\mathbf{a}\beta)
& =
D(\alpha\mathbf{a}\beta)-(\alpha\mathbf{a}\beta)\cdot F_1+F_1\cdot(\alpha\mathbf{a}\beta)\\
& =
\Delta(\alpha\mathbf{a}\beta)-(\alpha\mathbf{a}\beta)\cdot F_1+F_1\cdot(\alpha\mathbf{a}\beta)\\
& =
\Delta(\alpha)\cdot(\mathbf{a}\beta)+\alpha\cdot\Delta(\mathbf{a})\cdot\beta+(\alpha \mathbf{a})\cdot\Delta(\beta)\\
&\quad {}-(\alpha\mathbf{a}\beta)\cdot F_1+F_1\cdot(\alpha\mathbf{a}\beta)\\
& =
(\alpha\cdot F_1-F_1\cdot\alpha)(\mathbf{a}\beta)
+\alpha\cdot D(\mathbf{a})\cdot\beta
+(\alpha \mathbf{a})\cdot (\beta\cdot F_1-F_1\cdot\beta)\\
&\quad {}-(\alpha\mathbf{a}\beta)\cdot F_1+F_1\cdot(\alpha\mathbf{a}\beta)\\
& =
\alpha\cdot F_1\cdot(\mathbf{a}\beta)
+\alpha\cdot D(\mathbf{a})\cdot\beta
-(\alpha\mathbf{a})\cdot F_1\cdot\beta \\
& =
\alpha\cdot\bigl(
D(\mathbf{a})-\mathbf{a}\cdot F_1+F_1\cdot\mathbf{a}
\bigr)\cdot\beta \\
& =
\alpha\cdot D_0(\mathbf{a})\cdot\beta.
\end{split}
\end{equation*}

We now define a continuous linear map $d\colon A\to A^*$ by
\[
d(a)(b)=D_0(aE_{11})(bE_{11}) \quad  (a,b\in A).
\]
It is clear that $d$ cyclic. The task is now to show that $d$ is a derivation.
Let $a$, $b$, $c\in A$. Then
\begin{equation*}
\begin{split}
d(ab)(c)
& =
D_0\bigl((ab)E_{11}\bigr)(cE_{11})
=
D_0\bigl((aE_{11})(bE_{11})\bigr)(cE_{11})\\
& =
\bigl(D_0(aE_{11})\cdot(bE_{11})+(aE_{11})\cdot D(bE_{11})\bigr)(cE_{11})\\
& =
D_0(aE_{11})\bigl((bE_{11})(cE_{11})\bigr)+D(bE_{11})\bigl((cE_{11})(aE_{11})\bigr)\\
& =
D_0(aE_{11}) \bigl( (bc)E_{11} \bigr) +D(bE_{11})\bigl((ca)E_{11}\bigr)\\
& =
d(a)(bc)+d(b)(ca)\\
& =
\bigl(d(a)\cdot b+a\cdot d(b)\bigr)(c).
\end{split}
\end{equation*}
Since $d$ is a continuous cyclic derivation and $A$ is cyclically amenable, it follows that there exists
$f\in A^*$ such that
\[
d(a)=a\cdot f-f\cdot a \quad  (a\in A),
\]
\emph{i.e.},
\[
d(a)(b)=f(ba-ab) \quad  (a,b\in A).
\]
We now define $F_2\in M_n(A)^*$ by
\[
F_2(\mathbf{a})=\sum_{k=1}^{n} f(a_{kk}) \quad  (\mathbf{a}=(a_{ij})\in M_n(A)).
\]
We claim that
\begin{equation}\label{pp3}
D_0(\mathbf{a})=\mathbf{a}\cdot F_2-F_2\cdot\mathbf{a} \quad  (\mathbf{a}\in M_n(A)).
\end{equation}
Let $\mathbf{a}=(a_{ij})$, $\mathbf{b}=(b_{ij})\in M_n(A)$.
Taking into account~\eqref{mu1} and~\eqref{pp1}, we have
\begin{equation*}
\begin{split}
D_0(\mathbf{a})(\mathbf{b})
& =
D_0\left(\sum_{i,j=1}^n a_{ij}E_{ij}\right)(\mathbf{b})
=
\sum_{i,j=1}^n D_0(a_{ij}E_{ij})(\mathbf{b})\\
& =
\sum_{i,j=1}^n D_0\bigl(E_{i1}(a_{ij}E_{11})E_{1j}\bigr)(\mathbf{b})
=
\sum_{i,j=1}^n \bigl(E_{i1}\cdot D_0(a_{ij}E_{11})\cdot E_{1j}\bigr)(\mathbf{b})\\
& =
\sum_{i,j=1}^nD_0(a_{ij}E_{11})(E_{1j}\mathbf{b}E_{i1})
=
\sum_{i,j=1}^nD_0(a_{ij}E_{11})(b_{ji}E_{11})\\
& =
\sum_{i,j=1}^nd(a_{ij})(b_{ji})
=
\sum_{i,j=1}^nf(b_{ji}a_{ij}-a_{ij}b_{ji})\\
& =
F_2(\mathbf{b}\mathbf{a}-\mathbf{a}\mathbf{b}).
\end{split}
\end{equation*}
According to~\eqref{pp2} and~\eqref{pp3}, we have
\[
D(\mathbf{a})=\mathbf{a}\cdot F-F\cdot\mathbf{a} \quad  (\mathbf{a}\in M_n(A)),
\]
where $F\in M_n(A)^*$ is defined by $F=F_1+F_2$, which proves that $D$ is inner as required.
\end{proof}

\begin{theorem}\label{cj4}
Let $A$ be a cyclically amenable Banach algebra with a bounded approximate identity, and $n\ge 2$. Then $M_n(A)$ is a zero Lie product determined Banach algebra.
\end{theorem}

\begin{proof}
By Corollary~\ref{cj2} and Lemma~\ref{lj3},
$M_n(A)$ is cyclically amenable and every continuous Jordan derivation from $M_n(A)$ into ${M_n(A)}^{*}$ is a derivation. Corollary~\ref{rtja}  completes the proof.
\end{proof}

At this point it seems appropriate to mention a purely algebraic
result from~\cite{BGS}, stating that if a unital algebra $A$ is
zero Lie product determined, then so is $M_n(A)$. (The definition of zero Lie product algebra is the same as that of zero Lie product Banach algebra, only  the continuity is, of course, neglected.)

\begin{example}
Theorem~\ref{cj4} applies in any of the following cases.
\begin{enumerate}
\item
$A$ is a unital Banach algebra unitally polynomially generated by a single element.
\item
$A=L^1(G)$ where $G$ is any locally compact group.
\item
$A$ is a $C^*$-algebra. However, in this case $M_n(A)$ is itself a $C^*$-algebra
and therefore Corollary~\ref{wac} is applicable.
\end{enumerate}
\end{example}

\begin{remark}
Let $A$ be a unital Banach algebra, and $n\ge 2$.
By~\cite[Proposition~4.2]{ABESV}, the Banach algebra $M_n(A)$ is generated by idempotents
and~\cite[Example~1.3.3 and Theorem~2.11]{ABEV0} then show that it has property $\mathbb{B}$.
According to \cite[Theorem 2.7(iii)]{DLS}, the Banach algebra $M_n(A)$ is weakly amenable if and
only if $A$ is weakly amenable.
Consequently, if $A$ is cyclically amenable but not weakly amenable (such as $A(\mathbb{D})$, $C^n([a,b])$, $AC([a,b])$),
then the algebra $M_n(A)$ is not weakly amenable and cannot be handled by using Corollary~\ref{wa} but Theorem~\ref{cj4} works.
\end{remark}


\begin{thebibliography}{99}


\bibitem{AL}
M. Akkar, M. Laayouni.
Th\'eor\`emes de factorisation dans les alg\`ebres completes de Jordan.
\emph{Collect. Math.} \textbf{46} (1995), 239--254.


\bibitem{ABEV0}
J. Alaminos, M. Bre\v sar, J. Extremera, A.\,R. Villena,
Maps preserving zero products,
\emph{Studia Math.} \textbf{193} (2009), 131--159.


\bibitem{ABEVJordan}
J. Alaminos, M.  Bre\v sar, J. Extremera, A.\,R. Villena,
Characterizing Jordan maps on $C^*$-algebras through zero products,
\emph{Proc. Edinburgh Math. Soc.} \textbf{53} (2010),  543--555.

\bibitem{ABEVi}
J. Alaminos, M. Bre\v sar, J. Extremera, A.\,R. Villena,
Zero Lie product determined Banach algebras,
\emph{Studia Math.} \textbf{239} (2017), 189--199.

\bibitem{ABESV}
J. Alaminos, M. Bre\v sar, J. Extremera, \v{S}, \v{S}penko, A.\,R. Villena,
Commutators and square-zero elements in Banach algebras,
\emph{Quart. J. Math.} \textbf{67} (2016), 1--13.


\bibitem{ABV}
J. Alaminos, M. Bre\v sar,  A.\,R. Villena,
The strong degree of von Neumann algebras and the structure of Lie and Jordan derivations,
\emph{Math. Proc. Cambridge Philos. Soc. } \textbf{137} (2004),  441--463.

\bibitem{Bla}
A. Blanco,
On the weak amenability of $\mathcal{A}(X)$ and its relation with the approximation property,
\emph{J. Funct. Anal.} \textbf{203} (2003), 1--26.

\bibitem{Bla2}
A. Blanco,
On the weak amenability of $\mathcal{B}(X)$,
\emph{Studia Math.} \textbf{196} (2010), 65--89.



\bibitem{Br} M. Bre\v sar,
Jordan derivations revisited,
\emph{Math. Proc. Cambridge Philos. Soc.} \textbf{139} (2005),  411--425.

\bibitem{BGS}
M. Bre\v sar, M. Gra\v si\v c, J. Sanchez,
Zero product determined matrix algebras,
\emph{Linear Algebra Appl.} \textbf{430} (2009), 1486--1498.

\bibitem{D}
H.G. Dales,
\emph{Banach algebras and automatic continuity,}
London Mathematical Society Monographs, New Series, 24, Oxford Science Publications,
The Clarendon Press, Oxford University Press, New York, 2000.

\bibitem{DGG}
H. G. Dales, F. Ghahramani, and N. Gr\o nb\ae{}k,
Derivations into iterated duals of Banach algebras,
\emph{Studia Math.} \textbf{128} (1998), 19--54.

\bibitem{DLS}
H. G. Dales, A. T.-M. Lau, D. Strauss,
Banach algebras on semigroups and on their compactifications.
\emph{Mem. Amer. Math. Soc.} \textbf{205} (2010), no. 966, vi+165 pp.

\bibitem{G}
S. Goldstein,
Stationarity of operator algebras,
\emph{J. Funct. Anal.} \textbf{118}  (1993),   275--308.

\bibitem{Gr1}
N. Gr\o{}nb\ae{}k,
Weak and cyclic amenability for noncommutative Banach algebras,
\emph{Proc. Edinburgh Math. Soc.} \textbf{35} (1992), 315--328.

\bibitem{Gr2}
N. Gr\o{}nb\ae{}k,
Factorization and weak amenability of algebras of approximable operators,
\emph{Math. Proc. R. Ir. Acad.} \textbf{106A} (2006), 31--52.

\bibitem{Gr3}
N. Gr\o{}nb\ae{}k,
Bounded Hochschild cohomology of Banach algebras with a matrix-like structure,
\emph{Trans. Amer. Math. Soc.} \textbf{358} (2006), 2651--2662.

\bibitem{J}
B.\,E. Johnson,
Symmetric amenability and the nonexistence of Lie and Jordan derivations.
\emph{Math. Proc. Cambridge Philos. Soc.} \textbf{120} (1996),  455--473.


\bibitem{R}
C. J. Read,
Discontinuous derivations on the algebra of bounded operators on a Banach space,
\emph{J. London Math. Soc.} \textbf{40} (1989), 305--326.

\end{thebibliography}
\end{document}